\documentclass[10pt,a4paper]{article}
\usepackage[latin1]{inputenc}
\usepackage{amsmath}
\usepackage{amsfonts}
\usepackage{amssymb}
\usepackage{graphicx}
\usepackage{amsthm}
\usepackage{bbm}
\usepackage{nicefrac}
\usepackage{theoremref} 
\usepackage[title]{appendix}
\usepackage{hyperref}
\usepackage{dsfont}
\usepackage[
backend=bibtex,
style=alphabetic,
sortlocale=auto,
natbib=true,
url=false, 
doi=true,
eprint=false,
safeinputenc=true,
giveninits=true,
isbn=false,
maxbibnames=99
]{biblatex}
\usepackage{xcolor}
\numberwithin{equation}{section}
\newtheorem{Thm}{Theorem}[section]
\newtheorem{Lemma}{Lemma}[section]

\newtheorem{prop}{Proposition}[section]
\newtheorem{Def}{Definition}[section]
\newtheorem*{rmk}{Remark}
\newtheorem*{rmks}{Remarks}

\newcommand{\abs}[1]{\left|#1\right|}
\newcommand{\sgn}[1]{\operatorname{sgn}\left(#1\right)}

\usepackage[left=30mm,right=30mm]{geometry}

\begin{document}
	\title{The Ginibre Ensemble Conditioned on an \linebreak  Overcrowding Event}
	\author{Offer Kopelevitch\textsuperscript{1}}
	\footnotetext[1]{School of Mathematical
		Sciences, Tel Aviv University, Tel Aviv, 69978, Israel. E-mail: Ofermoshek@mail.tau.ac.il. This research was supported by ISF Grants 1903/18 and 3537/24.}
	\maketitle
	\begin{abstract}
		We look at the eigenvalues of the complex Ginibre Ensemble of random matrices consisting of size $N$. We study the event that for $ {c \in [0,1]}$, $\lfloor cN \rfloor$ of the eigenvalues are located outside of a disk of radius $ R \in (\sqrt{1-c},1)$. Except for the case $c=1$ the eigenvalue process conditioned on this event is not determinantal. Nevertheless we are able to obtain asymptotic estimates of the probability of the event, and describe the conditional distribution in three spatial regions. For $\{ \lambda \in \mathbb{C} : \abs{\lambda} <R\}, \{\lambda \in \mathbb{C} :\abs{\lambda} > R+\epsilon\} $ the conditional distribution is asymptotically that of a Ginibre ensemble. Meanwhile, near the boundary of the disk, after rescaling by a factor of order $ N$, it tends to the determinantal point process that appears in the limit of the Ginibre ensemble near a hard wall in Seo \cite{CoulombGasNearWall}.

	\end{abstract}
	\section{Introduction}
	Consider the Ginibre ensemble defined by the probability distribution on $ \mathbb{C}^N $
	
	\begin{equation} \label{Ginibre def}
		\frac{1}{Z_N}  \prod_{1\leq j < k \leq N} \abs{\lambda_k -\lambda_j}^2 \prod_{j=1}^{N}  e^{-N\abs{\lambda_j}^{2}}\mathrm{d}\lambda_j 
	\end{equation}	
	where $Z_N $ is a normalisation constant.
	
	The Ginibre ensemble is the most well studied two-dimensional determinantal point process in the field of random matrices \cite[Chapter 15]{LogGases}. It is the distribution of the eigenvalues of an $N \times N$  random matrix whose entries are independent complex Gaussian variables centred around zero with variance $\frac{1}{N} $. For a general review of the study of the Ginibre ensemble see \cite{GinibreBook}.
	
	The limiting distribution of the empirical measure of the Ginibre ensemble is the normalised Lebesgue measure on the unit disk (see Ginibre \cite{Ginibre}). In particular, for $ R \in (0,1) $, with probability one, $\frac{1}{N}\sum_{i=1}^N \mathds{1}_{\{ \abs{\lambda_i} >R \}} $ tends to $1- R^2$ as $ N $ tends to infinity. Moreover, according to the large deviation principle for the Ginibre Ensemble established by Petz and Hiai \cite{LargeDeviationsPH}, the event where the number of eigenvalues satisfying $\abs{\lambda}>R $ is $\lfloor cN \rfloor$ is a very rare event for $ c \in (1-R^2,1]$, with probability exponentially small in $N^2 $. 
	
	In 1993 the leading order of the overcrowding probability in different ranges of values of $ {c - \left(1-R^2\right)}$ was predicted by Jancovici, Lebowitz and Manificat \cite{JLM} and in 2014 Allez, Touboul and Wainrib  got the same prediction by a different approach \cite{OvercrowdingGinibreLeadingOrder}. In 2009 Akemann, Phillips and Shifrin found an exact formula for the probability that exactly $k$ eigenvalues are bigger than $R $ in absolute value. In our paper we study the case of $ k =\lfloor cN \rfloor $ with $ c\in (1-R^2,1] $ constant and look at the limit as $ N $ tends to infinity.

	The special case of the hole event, corresponding to $c=1$ was studied before, for example in \cite{GapProbabilityInPhysics}. The Ginibre ensemble conditioned on this event forms a determinantal point process and there are several results about it including the leading orders of the asymptotics of the probability of that event by Charlier \cite{LargeGapAsymptotics} and the leading order of the hole probability in a general open set by Adhikari and Reddy \cite{OpenSetHoleProbability}. Those results were generalised to other ensembles in \cite{BetaHoleProbability} by Adhikari and in \cite{CoulombGasesHoleProbabilities} by Charlier. Meanwhile for the infinite Ginibre ensemble the leading order of the hole probability was found by Forrester in \cite{SmallNumberOfEVInInfiniteGinibre}. A more detailed history of the results on the hole event appears in \cite{LargeGapAsymptotics}, see also Chapter 1.3 of \cite{GapProbabilityRiemannSphere}.
	
	The overcrowding cases, where $ c \in (1-R^2,1)$, and similarly the undercrowding cases with $ c \in (0,1-R^2) $, are not determinantal. Thus, in this paper, we take a different approach to deal with it and we get the asymptotics of the probability of the overcrowding event.
	\begin{Thm} \label{overcrowding probability}
		Let $\mathcal{G}_N$ be the Ginibre ensemble of size $N$ and assume $ R^2 >1-c $. The probability that exactly $N_c = \lfloor cN \rfloor $ eigenvalues of a size $N$ Ginibre ensemble are bigger in absolute value than $R$ satisfies 
		
		\begin{align*}
			\mathbb{P}_{\Lambda \sim \mathcal{G}_N} (\# \{\lambda \in \Lambda: &\abs{\lambda} >R\} = N_c) = \\
			&\mathbb{P}_{\mathcal{G}_{N_{c}}} (\forall \lambda \colon \abs{\lambda}>R ) \cdot 	\sum_{l=0}^{\infty} p(l)\left(\frac{R^2}{1-c}\right)^{-l} \left(1+O\left(\frac{\log^{3}N}{N}\right)\right)
		\end{align*}
		where $p(l) $ is the partition function (see Appendix~\ref{partition function}).
	\end{Thm}	
	
	Note that using this theorem, by substituting the asymptotic expansion of $\mathbb{P}_{\mathcal{G}_{N_{c}}} (\forall \lambda \colon \abs{\lambda}>R ) $ found by Charlier \cite[Theorem 1.7]{LargeGapAsymptotics}, we can write an asymptotic estimate to the probability as 
	\begin{align*}
		\mathbb{P}_{\mathcal{G}_N} (\# \{\lambda: &\abs{\lambda} >R\} = N_c)  = \\
		&\exp\left(C_1 N^2 + C_2 N \log N + C_3 N + C_4 \sqrt{N} + C_5 \log N + C_6 + O\left(N^{-\frac{1}{12}}\right) \right)
	\end{align*} 
	for explicit values of $C_1,C_2,C_3,C_4,C_5,C_6$. 
	
	Next we study the limits of the conditional distributions given the overcrowding event. Figure~\ref{fig:GinibreSim} shows an illustration of the conditional distribution of the eigenvalues of a Ginibre matrix of size $N=4000$, conditioned on the event that $3600$ of them (corresponding to $c=0.9 $) satisfy $\abs{\lambda}\geq R=0.7$. Three regions are visible: the interior region $|\lambda|<R$, a narrow boundary layer around $|\lambda|=R$, and the exterior region $|\lambda|>R+\epsilon$.
	
	\begin{figure}[h]
		\centering
		\includegraphics[width=0.7\linewidth]{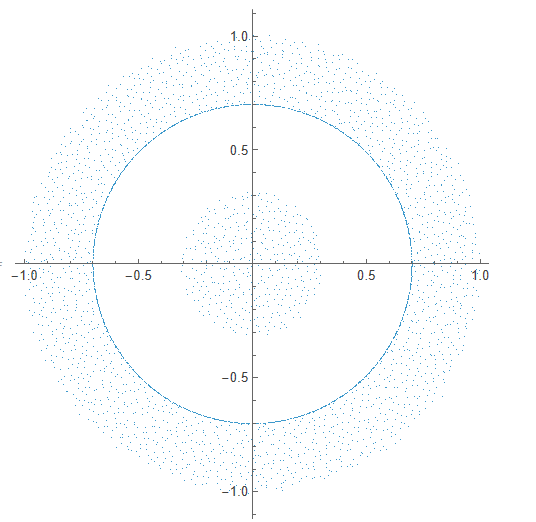}
		\caption[The conditional distribution of the eigenvalues]{Illustration of the conditional distribution of the eigenvalues of the Ginibre ensemble of size 4000, conditioned on the event that 90\% of them satisfy $\abs{\lambda}\geq R=0.7$, showing the three different regions of the distribution: inside the disk of radius 0.7, on the boundary of the disk and outside the disk.}
		\label{fig:GinibreSim}
	\end{figure}
	
	The main results of this paper are the limits of the conditional distribution. to phrase them we first define the following term.
	
	\begin{Def} \label{asymptotic similarity def}
		For two series of point process $ A_N , B_N$ we say that $A_N$ behaves asymptotically as $B_N$ if for any natural number $k$ and any Lipschitz function with compact support $f \colon \mathbb{C}^k \to \mathbb{C}$ we have
		\begin{equation*}
			\mathbb{E}_{X\sim A_N}\left(\sum_{x_1,\dots,x_k \in X} f(x_1,\dots,x_k) \right) = \mathbb{E}_{X\sim B_N}\left(\sum_{x_1,\dots,x_k \in X} f(x_1,\dots,x_k) \right) +o(c_N^k)
		\end{equation*}
		
		where $c_N = \mathbb{E}_{X\sim B_N} (\# \left( X \cap D(0,1) \right))$  is representing the intensity of the process.
	\end{Def}
	
	 Now we can state our results regarding the limits of the conditional distribution in three different regions.
	
	\begin{Thm} \label{conditional distribution summary}
		For $c \in (1-R^2,1] $, the conditional distribution of the Ginibre Ensemble of size $N$ given that $N_c =\lfloor cN \rfloor $ of them satisfy $\abs{\lambda} >R $ behaves asymptotically, as $N\to \infty $, in the following way:
		\begin{enumerate}
			\item Inside the disk $D(0,R)$, it is the same as a Ginibre ensemble of size $N-N_c$ scaled by $\sqrt{\frac{N-N_c}{N}}$.
			\item For any $\epsilon >0 $, outside the disk the disk $D(0,R+\epsilon)$ it is the same as a Ginibre ensemble of size $N$.
			\item Near the edge of the disk $D(0,R) $, after zooming in around $ R$ by a factor of $ \frac{N(R^2-1+c)}{R}$, it is the same as the determinantal point process with kernel 
			\begin{equation*}
				K_{\mathcal{X}}(z,w)=\frac{1-(z+\overline{w}+1) e^{-z-\overline{w}}} {\pi (z+\overline{w})^2} 
			\end{equation*}
			with respect to the Lebesgue measure on the right half plane.
		\end{enumerate}
		
	\end{Thm}
	\begin{rmks}
		\begin{enumerate}
			\item In particular this result imply that the $k$-point functions of the two processes are asymptotically the same.
			\item The intensity of the processes, $c_N$ in Definition \ref{asymptotic similarity def}, is of order $N$ for regions 1 and 2 and of order $1$ for region 3.
			\item Theorem \ref{conditional distribution summary} follows from Proposition~\ref{Ginibre regions} and Proposition~\ref{hard wall limit distribution} which include estimates for the rates of convergence.
			\item According to the rotational symmetry of the problem, when zooming in around a point $Re^{i \theta} $ by the factor of $\frac{N(R^2-1+c)}{R}$, the limiting process is the same process $ \mathcal{X}$ rotated by $\theta $.
			\item The process $\mathcal{X} $, in part 3 of Theorem \ref{conditional distribution summary}, appears in Seo \cite{CoulombGasNearWall} as the limit of a Ginibre ensemble with a hard wall inside the bulk near the hard wall. Cronvall and Wennman showed the universality of that process near a hard edge \cite{HardEdgeUniversality}.
		\end{enumerate}
	\end{rmks}

	The paper is organised as follows: In Section \ref{Eigenvlaue mixture probabilities} we use a result from \cite{GAFBook} to represent the distribution of the eigenvalues as a mixture of determinantal point precesses and show which of those processes have significant probability to appear. In Section \ref{Overcrowding Probability} we use the previous result to prove Theorem~\ref{overcrowding probability}. Then, in Section~\ref{CondDistSection}, we prove Theorem~\ref{conditional distribution summary}.
	
	\subsection{Additional and related results}
	The probability of an overcrowding event in similar processes was studied before, for example, a result from the physics literature by Majumdar, Nadal, Scardicchio and Vivo calculated the leading order asymptotics of the probability of overcrowding of eigenvalues in the positive half-line in the Gaussian unitary ensemble (GUE) \cite{PositiveEigenvaluesOfGUE}. As far as we know the next order term is only known for the hole event (see Dea\~no and Simm \cite{PositiveDefiniteGUE}).
	
	For the infinite Ginibre ensemble, the leading term of the overcrowding probability was shown by Shirai \cite{InfiniteGinibreOvercrowdingLDP} and Fenzl and Lambert found the next terms in the asymptotics of the overcrowding probability \cite{InfiniteGinibreOvercrowdingAsymptotics} which correspond to the Jancovici, Lebowitz, Manificat (JLM) predictions \cite{JLM}. The overcrowding event on general sets in the finite Ginibre ensemble was studied by Armstrong, Serfaty and Zeitouni who found properties the limiting distribution of the eigenvalues must satisfy \cite{OvercrowdingInGeneralSet}. In our case of overcrowding outside the disk $D(0,R) $ their result shows the existence of a singular part of the limiting measure supported on $\{z\in \mathbb{C} : \abs{z}=R \} $, and  that for some $\epsilon>0 $ the limiting measure has zero measure on $ \{z\in \mathbb{C} : \abs{z}\in [R-\epsilon,R) \}$, in accordance with our Theorem \ref{conditional distribution summary}.
	
	The main results of this paper can be generalised with minor modifications to the probability  distribution, studied by Charlier \cite{LargeGapAsymptotics},
		\begin{equation}
			\frac{1}{Z_N^{\alpha,b}} \prod_{1\leq j < k \leq N} \abs{\lambda_k -\lambda_j}^2 \prod_{j=1}^{N} \abs{\lambda_j}^{2\alpha} e^{-N\abs{\lambda_j}^{2b}}\mathrm{d}\lambda_j, \qquad b>0, \alpha >-1 
		\end{equation}
	where $Z_N^{\alpha,b} $ is a normalisation constant. We show them only in the case of the Ginibre distribution, corresponding to $ \alpha = 0, b=1 $, for ease of computation. In addition, our method could be extended to results on the conditional distribution given an overcrowding event in other radially symmetric sets, such as unions of disks and annuli centred around the origin. In those cases the main flow of the proof remains unchanged while the technical lemmas require more precise estimates and bounds on the incomplete gamma function.
	
	An interesting question that would require different tools is the question of overcrowding on a non-radially-symmetric set. In that case the orthogonal polynomials on the set will no longer be monomials and so it is unclear how to have a nice representation of the functions $\phi_k$ in \eqref{kernel of Ginibre outside disk}. 
	
	\paragraph{Notations}
		In this paper we will use the following notations:
		\begin{itemize}
			\item The incomplete gamma functions
			\begin{align*}
				\Gamma(a,s) = \int_{s}^{\infty} x^{a-1}e^{-x}\mathrm{d}x, && 			\gamma(a,s) = \int_{0}^{s} x^{a-1}e^{-x}\mathrm{d}x.
			\end{align*}
			\item The normalised incomplete gamma function 
			\begin{equation*}
				Q(a,z)=\frac{\Gamma(a,z)}{\Gamma(a)}.
			\end{equation*}
			See Appendix~\ref{incomplete gamma appendix} for details.
			\item $\mathcal{G}_{N} $ denotes the Ginibre ensemble of size $N$.
			\item $ \mathcal{G}_{N,c,R} $ denotes the Ginibre ensemble of size $N$ conditioned of the event that $N_c = \lfloor cN \rfloor $ eigenvalues satisfy $\abs{\lambda}>R $.
			
			\item For a point process $\mathcal{Y}$, $\mathcal{Y}^{\wedge k} $ denotes the point process defined by choosing $ k $ random points from $\mathcal{Y} $ without repetition.
			\item $D(z,r)$ for the disk in the complex plane centred at $z$ with radius $r$.
			\item $\mathbb{C}_+ = \{ z\in \mathbb{C} : \Re(z) >0 \} $ is the right half plane.
		\end{itemize}

	\paragraph{Acknowledgement} I would like to thank Aron Wennman for his insight that the Ginibre ensemble is well-suited to study the overcrowding event due to its radial symmetry. I would also like to thank Christophe Charlier and Arno Kuijlaars for the helpful discussions with them about this problem. I am grateful to Sasha Sodin and Ofer Zeitouni for their suggestions and remarks leading to improvements of the error term in the main result and for the anonymous referees for improving the readability of the paper and the description of the background of the problem.  Finally I would like to thank my advisor Alon Nishry for his continuous help and suggestions during this research.


	\section{Eigenvalue distribution inside and outside the disk} \label{Eigenvlaue mixture probabilities}
	
	We denote by $ \lambda_1,\dots, \lambda_N $ the eigenvalues of a Ginibre $N \times N $ random matrix, as in \eqref{Ginibre def} (see \cite{GAFBook}, \cite{Ginibre}). They form a determinantal point process with kernel
	\begin{equation}
		K_N(z,w)=\sum_{k=0}^{N-1} \frac{N^{k+1}(z\overline{w})^k e^{-\frac{1}{2}N(\abs{z}^2+\abs{w}^2)}}{\pi k!} .
	\end{equation}
	We start by noting that if we look at the process restricted to the complex plane outside a disk of radius $ R$, i.e. at the random set $ \left\{ \lambda_i \colon i\in \{1,\dots,N\}, \abs{\lambda_i}>R \right\} $ it is still a determinantal point process with the same kernel restricted to the set $ \mathbb{C} \backslash D(0,R) $.
	
	Outside the disk $D(0,R) $ the functions $ \phi_{k}(z) = \sqrt{\frac{N^{k+1}}{\pi \Gamma(k+1,NR^2)}}e^{-\frac{1}{2}N\abs{z}^2}z^k $ for $k=0,1,\dots $ are orthonormal and we can write the kernel as
	\begin{equation} \label{kernel of Ginibre outside disk}
		 K_N(z,w) = \sum_{k=0}^{N-1} a_k \phi_{k}(z)\overline{\phi_{k}(w)}
	\end{equation}
	with $ a_k=\frac{\Gamma(k+1,NR^2)}{k!} $.
	Thus, as a special case of \cite[Theorem 4.5.3]{GAFBook}.
	
	\begin{prop} \label{unconditioned set probability}
		
		The distribution of the Ginibre eigenvalues that satisfy $\abs{\lambda} >R $ is a mixture of determinantal point processes indexed by sets $ J \subseteq \{0,1,\dots,N-1\}$ with kernels 
		\begin{equation} \label{J kernel}
			K^J(z,w)=\sum_{k\in J} \phi_{k}(z)\overline{\phi_k(w)} \mathds{1}_{\{\abs{z},\abs{w}>R\}} .
		\end{equation}

		The probability of the process with index $J$ is 
		\begin{equation} \label{unconditional probability}
			\tilde{P}_J = \prod_{k=0}^{N-1} (1-a_k) \prod_{k\in J}  \frac{a_k}{1-a_k}.
		\end{equation}
	\end{prop}
	
	\begin{rmk}
		In this case a mixture of point processes with $k$-point functions $f^{(k)}_J$ with probabilities $p_J$ can be defined as the point process with $k$-point functions 
		\begin{equation*}
			\sum_{J} p_J \cdot f^{(k)}_J(x_1,\dots, x_k).
		\end{equation*}
		As in our case the total number of points is bounded it determines the point process uniquely.
	\end{rmk}
		
	This proposition shows that the process outside the disk is a mixture of projection point processes (and similarly inside the disk). The amount of points in each of those processes is deterministically $\#J $. Moreover, the overcrowding event depends only on the configuration of points outside the disk (in particular on their number) and thus conditioning the process on the overcrowding event and then restricting it to $\mathbb{C} \backslash D(0,R) $ is the same as restricting the process to $\mathbb{C} \backslash D(0,R) $ and then conditioning on the overcrowding event. Thus, we conclude the next Proposition.
	
	\begin{prop} \label{conditioned mixture}
		The distribution of the Ginibre eigenvalues that satisfy $\abs{\lambda}>R $ conditioned on the event that there are exactly $m$ such eigenvalues is a mixture of of determinantal point processes indexed by $ J \subseteq \{0,1,\dots,N-1\}$ with $ \#J = m $ with kernels as in \eqref{J kernel}. The probability of the process with index $J $ is 
		\begin{equation} \label{Set prob conditioned}
			P_J = L_m \prod_{k\in J}  \frac{a_k}{1-a_k}
		\end{equation}
		where $ L_m$ is a constant not depending on $J$.
	\end{prop} 
	
	Similarly defining 
	\begin{equation} \label{inner orthogonal pol def}
		 \tilde{\phi}_k(z) = \sqrt{\frac{N^{k+1}}{\pi \gamma(k+1,NR^2)}}e^{-\frac{1}{2}N\abs{z}^2}z^k
	\end{equation}
	 we have
	
	\begin{prop} \label{conditioned mixture inner}
		The distribution of the Ginibre eigenvalues that satisfy $\abs{\lambda } <R $ conditioned on the event that there are exactly $N-m$ such eigenvalues is a mixture of of determinantal point processes indexed by $ J \subseteq \{0,1,\dots,N-1\}$ with $ \#J = m $ with kernels
		\begin{equation} \label{J kernel inner}
			K^J(z,w)=\sum_{k\notin J} \tilde{\phi}_{k}(z)\overline{\tilde{\phi}_{k}(w)} \mathds{1}_{\{\abs{z},\abs{w}<R\}} .
		\end{equation}
		The probability of the process with index $J $ is given by \eqref{Set prob conditioned}.
	\end{prop}

	From now on we will look at the parameters $R,c $ as constant and be interested in the asymptotics as $N$ goes to infinity. Moreover we assume the case of overcrowding outside the disk, i.e. $ 0<R<1, c>1-R^2 $. The case of undercrowding and the case $ R>1 $ are similar.
	We will denote by $ \mathbb{P}_c $ the probability space in Proposition~\ref{conditioned mixture}, i.e. the mixture of different values of $J$ conditioned on the event $ \# J =N_c $.
	
	\pagebreak
	In the rest of this section we will estimate the probability $P_J$ for such a set $ J$ of size $N_c = \lfloor cN \rfloor$. 
	We will use the following notation.
	\begin{Def}	
		Let 
		\begin{equation*}
			n : \left\{ J \subseteq \{1,\dots , N\} : \#J =N_c \right\} \to \left\{\mathbf{n} \in \mathbb{Z}^{N_c} : N-N_c \geq n_{1} \geq n_{2} \geq\dots \geq n_{N_c} \geq 0 \right\}
		\end{equation*}
		be the bijection defined by $n(\{x_1,\dots,x_{N_c}\}) = \left(n_1,n_2,\dots,n_k\right)$ where ${x_1 < x_2 < \dots < x_{N_c}}$ and  $ {n_k = N-N_c +k-x_k} $. 
		
		For such a vector $ \mathbf{n}$, we denote by $J_{\mathbf{n}} $ the set satisfying $n(J_{\mathbf{n}}) = \mathbf{n} $.
		
		In the rest of the paper when we write sum over $\mathbf{n} $, we will mean that the sum is over the set $\left\{\mathbf{n} \in \mathbb{Z}^{N_c} : N-N_c \geq n_{1} \geq n_{2} \geq\dots \geq n_{N_c} \geq 0 \right\} $. In addition we denote $P_{J_\mathbf{n}}=P_{\mathbf{n}}, \tilde{P}_{J_\mathbf{n}}=\tilde{P}_{\mathbf{n}}$. 

	\end{Def} 
	
	By \eqref{unconditional probability} we have 
	\begin{equation} \label{Pn formula} 
		\tilde{P}_{\mathbf{n}} = \prod_{k \in A} Q(k,NR^{2}) \cdot \prod_{k \notin A} \left(1-Q(k,NR^{2})\right) .
	\end{equation}

	Above $ Q(a,z) = \frac{\Gamma(a,z)}{\Gamma(a)} $ is the normalised incomplete gamma function (see Appendix \ref{incomplete gamma appendix}), and in particular 
	\begin{equation} \label{probability quatient}
		\frac{\tilde{P}_{\mathbf{n}}}{\tilde{P}_{\mathbf{0}}} =\frac{P_{\mathbf{n}}}{P_{\mathbf{0}}} =
		\prod_{k=1}^{N_c} \frac{Q(N-N_c+k-n_k, NR^{2} )}{Q(N-N_c+k, NR^{2} )} \frac{1 - Q(N-N_c+k, NR^{2} )}{1- Q(N-N_c+k-n_k, NR^{2} )} .
	\end{equation}

	To find the values of $n(J) $ with significant probability and to estimate $ P_{\mathbf{n}} $ for them we prove the following series of lemmas:
	
	\begin{Lemma} \label{first bound}
		There exist $ \epsilon>0 $ and $\delta >0 $ such that for $N $ big enough
		\begin{equation*}
			\mathbb{P}_c \left(\exists x\in \{1,\dots,N\}\backslash J : x>NR^2 (1-\delta) \right)  \leq e^{-\epsilon N} P_{\mathbf{0}} .
		\end{equation*} 
	
	\end{Lemma}	

	\begin{rmk}
		This Lemma in implies that there exists $\delta>0 $ such that conditioned on the overcrowding event, the probability that there exists $x>NR^2 (1-\delta) $, $x \notin J $ is exponentially small in $N $. 
	\end{rmk}
	
	\begin{proof}[Proof of Lemma~\ref{first bound}]
		
	We will choose $\delta >0 $ such that $R^2(1-\delta) +c-1 > 0 $ and $N  $ big enough such that $ D= \lfloor {NR^2(1-\delta)} \rfloor +N_c -N > 0 $. Notice that if there exists $ x \in \{1,\dots,N\} \backslash J$ such that $ x>NR^{2}{(1-\delta)}$ then 
	for $k=1,\dots, D $ the $k$-lowest term in $J$ is less than the $k$-lowest term in $J_{\mathbf{0}}$ so we have
	\begin{equation*}
		\sum_{k=1}^{D} x_k \leq \sum_{k=1}^{D} \left[N -N_c +k\right] - D .
	\end{equation*}
	 Hence 
	
	\begin{equation} \label{BoundSumNk}
		\sum_{k=1}^{D} n_k = \sum_{k=1}^{D} (N-N_c+k-x_k) \geq  D .
	\end{equation}
	
	We use Lemma~\ref{incomplete gamma recursion} to get
	\begin{align*}
		\frac{Q(n+1,NR^2)}{Q(n,NR^2)} = \frac{\Gamma(n+1,NR^2)}{\Gamma(n,NR^2)} \frac{\Gamma(n)}{\Gamma(n+1)} \\
		= \frac{n\Gamma(n,NR^2)+\left(NR^2\right)^n e^{-n}}{n\Gamma(n,NR^2)}
	\end{align*}
	Hence for $ n < NR^2 (1-\delta) $, as we have 
	\begin{align*}
		\Gamma(n,NR^2) &= \int_{NR^2}^{\infty} t^{n-1} e^{-t} \mathrm{d}t \\
		&\leq \left(NR^2\right)^{n-1} e^{-NR^2} \int_{0}^{\infty} e^{-\left(1-\frac{n-1}{NR^2}\right)t}\mathrm{d}t \\
		&= \frac{\left(NR^2\right)^{n-1} e^{-NR^2}}{1-\frac{n-1}{NR^2}} ,
	\end{align*}
	there exists $0< \kappa = \kappa(\delta) <1 $ such that 
	\begin{equation} \label{Incomplete Gamma ratio}
		\frac{Q(n,NR^2)}{Q(n+1,NR^2)} <\kappa .
	\end{equation}
	And in particular
	\begin{equation} \label{BoundQForSmalln}
		Q(n,NR^2) \leq \kappa^{NR^2(1-\delta)-n} Q(NR^2(1-\delta),NR^2) \leq  \kappa^{NR^2(1-\delta)-n} .
	\end{equation}

	Now, combining \eqref{BoundSumNk},  \eqref{BoundQForSmalln} and \eqref{probability quatient}, we get that for $ L \geq D$ :
	\begin{align*}
		\mathbb{P}_c\left( \sum_{k=1}^{D} (n(J))_k =L  \right) &\leq P_{\mathbf{0}}  \sum_{n_1,\dots, n_{D}\colon \sum_{k=1}^{D}n_k = L} \prod_{k=1}^{D} \frac{Q(N-N_c+k-n_k, NR^{2} )}{Q(N-N_c+k, NR^{2} )}  \\ 
		&\leq P_{\mathbf{0}} \cdot \kappa^{L} \cdot p(L)
	\end{align*}
	where $ p(n) $ is the partition function (see Appendix~\ref{partition function}).
	Hence, according to Lemma~\ref{partition function bound}, for some $ 0<\tilde{\kappa} <1 $ and for $ N $ big enough $ \mathbb{P} \left( \sum_{k=1}^{D} (n(J))_k =L\right) \leq P_{\mathbf{0}} \cdot \tilde{\kappa}^L$ so 
	\begin{equation*}
		\mathbb{P}_c \left(\exists x\in \{1,\dots,N\}\backslash J : x>NR^2 (1-\delta) \right) \leq e^{-\epsilon N} P_{\mathbf{0}} .
	\end{equation*}

	\end{proof}

	\begin{Lemma} \label{bound on big sum}
		There exists $M_0>0$ such that for each $ M>M_0 $, there is $N_0=N_0(M,c,R)$ such that for each $N >N_0$
		\begin{equation*}
			\mathbb{P}_c \left(\sum_k (n(J))_k >M \log(N) \right) \leq P_{\mathbf{0}} \cdot N^{-10} . 
		\end{equation*}
	\end{Lemma}
	
	\begin{proof} 
		By Lemma~\ref{first bound} we have that except from a set of probability at most $ P_{\mathbf{0}} \cdot e^{-\epsilon N} $, for each $ k=1,\dots,N_c$ either $n_k=0 $ or $N-N_c+k< NR^{2}(1-\delta) $, i.e. outside this set for all $ {k> D=\lfloor {NR^2(1-\delta)} \rfloor +N_c -N} $, $ n_k =0$. Equivalently, $ n_{D+1}=0 $ . Hence we conclude that
		\begin{align*}
			\mathbb{P}_c&\left(\sum_{k=1}^{N_c} (n(J))_{k} >M\log(N)\right)   \\ 
			&\leq \mathbb{P}_c\left(\sum_{k=1}^{D} (n(J))_k > M \log(N) , (n(J))_{D+1}=0 \right) + e^{-\epsilon N}P_{\mathbf{0}} .
		\end{align*}
	
	Now, we have for $\mathbf{n} $ such that $  \sum_{k}^{D} n_{k} >M\log(N) $ and $n_{D+1}=0 $:
	\begin{align*}
		\frac{P_{\mathbf{n}}}{P_{\mathbf{0}}} &=  \prod_{k=1}^{\min(D, N_c)} \frac{ Q(N-N_c+k-n_k, NR^{2})} {Q(N-N_c+k, NR^{2})} \frac{1 - Q(N-N_c+k, NR^{2} )}{1- Q(N-N_c+k-n_k, NR^{2} )} \\
		&\leq \prod_{k=1}^{\min(D, N_c)} \frac{ Q(N-N_c+k-n_k, NR^{2})} {Q(N-N_c+k, NR^{2})} 
	\end{align*}
	and by \eqref{Incomplete Gamma ratio} we get 
	\begin{equation*}
		\frac{P_{\mathbf{n}}}{P_{\mathbf{0}}} \leq \kappa^{\sum_{k} n_k} .
	\end{equation*}
	Hence there exists $\tilde{\kappa}<1 $ such that for $ L \geq  M \log(N) $ and $ N $ big enough
	\begin{equation*}
		\frac{\sum_{\mathbf{n} \colon \sum n_k =L} P_{\mathbf{n}}}{P_{\mathbf{0}}} \leq \kappa^L p(L) \leq \tilde{\kappa}^L .
	\end{equation*}
	So choosing $ M_0 >\frac{10}{-\log(\tilde{\kappa})}$ we conclude that for big enough $N$ 
	\begin{equation*}
		\frac{\mathbb{P}_c\left(\sum_k (n(J))_k >M\log(N)\right)}{P_{\mathbf{0}}} = \sum_{\mathbf{n} \colon \sum_{k} n_{k}>M\log(N)}\frac{ P_{\mathbf{n}}}{P_{\mathbf{0}}} \leq \sum_{L=\lfloor M\log(N) \rfloor}^{\infty} \tilde{\kappa}^L \leq N^{-10}.
	\end{equation*}

	\end{proof}

	\begin{Thm} \label{set distribution}
		There exists $M_0 >0$ such that for each $M>M_0 $ there exists $K=K(M,c,R)>0$ such that the set of vectors, 
		\begin{equation*}
			\mathcal{B}_N = \left\{\mathbf{n} \colon \sum_k n_{k} \leq M\log(N)\right\}
		\end{equation*}
		satisfies that ${\mathbb{P}_c \left(n(J) \notin \mathcal{B}_N \right) \leq K P_{\mathbf{0}} \cdot  N^{-10}} $ and for $\mathbf{n}\in \mathcal{B}_{N}$ and $ k = \log\frac{R^{2}}{1-c}>0$:
		\begin{equation}
			\frac{P_{\mathbf{n}}}{P_{\mathbf{0}}} = e^{-k\sum_{j=1}^{N_c}n_j} \left(1+O\left(\frac{\log^{3}N}{N}\right)\right).
		\end{equation}
		
	\end{Thm}
	
	\begin{proof}
		We choose the $M_0$ of Lemma \ref{bound on big sum} and then for each $M>M_0 $ there exists a constant $ K>0 $ (depending only on $c,R,M $) such that the set  $ \mathcal{B}_N$ satisfies 
		\begin{equation*}
			\frac{\sum_{\mathbf{n}\notin \mathcal{B}_N} P_{\mathbf{n}}}{P_{\mathbf{0}}} \leq K N^{-10}.
		\end{equation*} 
		
		Then, for $ \mathbf{n} \in \mathcal{B}_N$ we have
		
		\begin{align*}
			\frac{P_{\mathbf{n}}}{P_{\mathbf{0}}} &= \prod_{j=1}^{N_c} \frac{Q(N-N_c+j-n_j, NR^{2} )}{Q(N-N_c+j, NR^{2} )} \cdot \frac{1 - Q(N-N_c+j, NR^{2}}{1- Q(N-N_c+j-n_j, NR^{2} )}\\
			&=\prod_{j=1}^{\lfloor M\log(N) \rfloor} \frac{Q(N-N_c+j-n_j, NR^{2} )}{Q(N-N_c+j, NR^{2} )} \cdot \frac{1 -Q(N-N_c+j, NR^{2} )}{1- Q(N-N_c+j-n_j, NR^{2} )}  .
		\end{align*}
		
		Since for $j<M \log(N) $ we have $N-cN+j-n_j \leq NR^2(1-\delta)$ for $N $ big enough and by Lemma~\ref{Incomplete Gamma estimates} we get:
		
		\begin{align*}
			Q(N-N_c +j -n_j,NR^2) &= \frac{e^{-NR^2+(N-N_c+j-n_j) (1+\log \frac{NR^2}{N-N_c+j-n_j})}}{\sqrt{2 \pi}} \cdot  \frac{1}{\frac{NR^2}{N-N_c+j-n_j} -1 } \cdot \\
			&\cdot  \frac{1}{\sqrt{N-N_c+j-n_j}} \left(1 + O(N^{-1})\right) \\
			&= A_{N,c,j} e^{(N-N_c+j-n_j)(1+\log \frac{NR^2}{N-N_c+j-n_j})} \left(1+O\left(\frac{n_j}{N}\right)\right) ,
		\end{align*}
		where $A_{N,c,j}$ does not depend on $n_j $. Hence, for some $\epsilon >0 $,

		\begin{align*}
			\frac{P_{\mathbf{n}}}{P_{\mathbf{0}}} &= \prod_{j=1}^{\lfloor M\log(N) \rfloor} e^{(N-N_c +j-n_j)(1+\log \frac{NR^2}{N-N_c+j-n_j}) - (N-N_c +j)(1+\log \frac{NR^2}{N-N_c+j})} \\
			&\cdot \left(1+O\left(\frac{n_j}{N}\right) + O(e^{-\epsilon N})\right) \\
			&= \prod_{j=1}^{\lfloor M\log(N) \rfloor} e^{-n_j +(N-N_c+j)\log \frac{N-N_c+j}{N-N_c+j-n_j} -n_j\log \frac{NR^2}{N-N_c+j-n_j}} \cdot \left(1+O\left(\frac{n_j}{N}\right)\right).
		\end{align*}	
		
		Using the estimate $ log(1+x) = x+ o(x^2)$ we get 	
		
		\begin{align*}
			\frac{P_{\mathbf{n}}}{P_{\mathbf{0}}} &=  \prod_{j=1}^{\lfloor M\log(N) \rfloor} e^{-n_j +(N-N_c+j)\left(\frac{n_j}{N-N_c+j} + O\left(\frac{n_j^2}{N^2}\right)\right)-n_j \log \left(\frac{R^2}{1-c} + O\left(\frac{j+n_j}{N}\right)\right)} \cdot \left(1+O\left(\frac{n_j}{N}\right)\right) \\
			&= \prod_{j=1}^{\lfloor M\log(N) \rfloor} e^{-n_j \log \frac{R^2}{1-c} + O\left(\frac{jn_j+n_j^2}{N}\right)} \cdot \left(1+O\left(\frac{n_j}{N}\right)\right) \\
			&= \left(\frac{1-c}{R^2}\right)^{\sum_j n_j} \left(1+O\left(\frac{\log^{3} N}{N}\right) \right) .
		\end{align*}

		This completes the proof.
	\end{proof}

	\section{Probability of the Overcrowding Event} \label{Overcrowding Probability}
	
	In this section we will estimate the probability of the overcrowding event 
	\begin{equation*}
		\mathbb{P}\left(\# \{\lambda\colon \abs{\lambda} >R \} = N_c \right) =\sum_{\#J=N_c} \tilde{P}_J .
	\end{equation*}
	Note that a similar formula appears in \cite[equation 3.16]{KGapProbabilities}.
	
	By Theorem \ref{unconditioned set probability} we have
	\begin{equation*}
		\tilde{P}_{\mathbf{0}} = \prod_{k=0}^{N-N_c}(1-a_k) \prod_{k=N-N_c+1}^{N} a_k
	\end{equation*}
	where $ a_k = \frac{\Gamma(k+1,NR^2)}{k!} =Q(k+1,NR^2)$. We begin by estimating the first product. Based on Lemma \ref{Incomplete gamma monotonicity} we have
	\begin{equation} \label{inner eigenvalues prob bound first}
		\prod_{k=0}^{N-N_c}(1-Q(k+1,NR^2)) \geq \left(1-Q(N-N_c+1,NR^2)\right)^{N-N_c+1} .
	\end{equation}
	
	Now, Lemma \ref{Incomplete Gamma estimates} implies that
	\begin{equation*}
		Q(N-N_c+1,NR^2) = \frac{e^{-NR^2+(N-N_c+1)(1+\log\frac{NR^2}{N-N_c+1})}}{\sqrt{2\pi (N-N_c+1)}(\frac{NR^2}{N-N_c+1}-1)} \left(1+O\left(\frac{1}{N}\right)\right).
	\end{equation*}
	So for some constants $A,\epsilon >0$ (that might depend on $c,R$) we have
	\begin{equation*}
		Q(N-N_c+1,NR^2) \leq A e^{-\epsilon N} .
	\end{equation*}
	Substituting into \eqref{inner eigenvalues prob bound first} we get 
	\begin{equation} \label{inner eigenvalues prob bound}
		1 \geq \prod_{k=0}^{N-N_c}(1-Q(k+1,NR^2)) \geq 1-AN e^{-\epsilon N} .
	\end{equation}
	Now we note that  
	\begin{equation} \label{outer eigenvalues prob estimate}
		 \prod_{k=N-N_c+1}^{N} Q(k+1,NR^2) = \mathbb{P}_{\mathcal{G}_{N_c}} (\forall \lambda \colon \abs{\lambda}>R ) .
	\end{equation}
	
	Thus we are left to estimate $ \frac{\mathbb{P}(\#J = N_c)}{\tilde{P}_{\mathbf{0}}} = \frac{\sum_{\mathbf{n}}P_{\mathbf{n}}}{P_{\mathbf{0}}} $.
	By Theorem \ref{set distribution} we have 
	\begin{equation*}
		\frac{\sum_{\mathbf{n}}P_{\mathbf{n}}}{P_{\mathbf{0}}} = \sum_{\mathbf{n}: \sum_j n_j \leq M\log N}\left(\frac{R^2}{1-c}\right)^{-\sum_{j=1}^{N}n_j} \left(1+O\left(\frac{\log^{3}N}{N}\right)\right) .
	\end{equation*}
	Hence, for $N$ big enough such that $N_c > \lfloor M \log N \rfloor $ we have
	\begin{align*}
		\mathbb{P}(\#J = N_c) &= P_{\mathbf{0}} \sum_{l=0}^{\lfloor M \log N \rfloor}\sum_{\mathbf{n}\colon \sum_{j=1}^{N} n_j =l} \left(\frac{R^2}{1-c}\right)^{-l} \left(1+O\left(\frac{\log^{3}N}{N}\right)\right) \\
		&= P_{\mathbf{0}} \sum_{l=0}^{\lfloor M \log N \rfloor } p(l) \left(\frac{R^2}{1-c}\right)^{-l} \left(1+O\left(\frac{\log^{3}N}{N}\right)\right)
	\end{align*}
	where $p(l) $ is the partition function [see Appendix~\ref{partition function}].
	Based on Lemma \ref{partition function bound}, for big enough values of $M \log N$ we have 
	\begin{equation*}
		\sum_{l= \lfloor M \log N \rfloor +1}^{\infty} p(l) \left(\frac{R^2}{1-c}\right)^{-l} \leq \left(\frac{R^2}{1-c}\right)^{-\frac{1}{2}M \log N}.
	\end{equation*}
	So, choosing $M $ big enough such that $ \frac{1}{2}M \log \left(\frac{R^2}{1-c}\right) >1$, we have
	\begin{equation} \label{ratio total prob to P0}
		\frac{\mathbb{P}(\#J = N_c)}{P_{\mathbf{0}}} = 	\sum_{l=0}^{\infty} p(l)\left(\frac{R^2}{1-c}\right)^{-l} \left(1+O\left(\frac{\log^{3}N}{N}\right)\right).
	\end{equation}
	
	Combining \eqref{inner eigenvalues prob bound}, \eqref{outer eigenvalues prob estimate} and \eqref{ratio total prob to P0} we conclude that
		\begin{align*}
			\mathbb{P} (\# \{\lambda: &\abs{\lambda} >R\} = N_c) = \\
			&\mathbb{P}_{\mathcal{G}_{N_c}} (\forall \lambda \colon \abs{\lambda}>R ) \cdot 	\sum_{l=0}^{\infty} p(l)\left(\frac{R^2}{1-c}\right)^{-l} \left(1+O\left(\frac{\log^{3}N}{N}\right)\right)
		\end{align*}
	which completes the proof of Theorem~\ref{overcrowding probability}.

	\section{Conditional distribution of eigenvalues} \label{CondDistSection}

	In this section we will use Theorem \ref{set distribution} to investigate the conditional distribution of the eigenvalues in three different regions, and show that with suitable rescaling it converges to determinantal point processes. Those are the region inside the disk, $\abs{z} < R $, the region outside of the disk not close to the boundary, $ \abs{z} >R+\epsilon $, where there is no rescaling required and the region outside the disk near the boundary where there is a normalisation of order $N$, i.e. $ \abs{z} = R+O(\frac{1}{N}) $. We begin with the first two cases.

	\begin{Thm} \label{Ginibre regions}
		For any natural number $k$, and any bounded Lipschitz function with compact support $f \colon \mathbb{C}^k \to \mathbb{C}$, the expectation $E_f = \mathbb{E}_{\{\lambda_1,\dots,\lambda_k\} \sim \mathcal{G}_{N,c,R}^{\wedge k} }f(\lambda_1,\dots,\lambda_k)$,  satisfies:
		\begin{enumerate}
			\item Suppose that $\operatorname{supp}f \subseteq \{z\in \mathbb{C} \colon \abs{z}<R\}^k $, then
			\begin{equation*}
				\abs{E_f - \mathbb{E}_{(\lambda_1,\dots,\lambda_k) \sim \mathcal{Y}^{\wedge k} }(f(\lambda_1,\dots, \lambda_k)) }  = O\left(\frac{\log N}{N}\right)
			\end{equation*}
			where $ \mathcal{Y}\sim \sqrt{\frac{N-N_c}{N}}\cdot \mathcal{G}_{N-N_c} $ is distributed as a Ginibre Ensemble of size $N-N_c $, scaled by $\sqrt{\frac{N-N_c}{N}} $.

			\item  Suppose that for some $\epsilon >0$, $\operatorname{supp}f \subseteq \{z\in \mathbb{C} \colon \abs{z}>R+\epsilon\}^k $, then
			\begin{equation*}
				\abs{E_f - \mathbb{E}_{\left(\lambda_1,\dots,\lambda_k\right)\sim \mathcal{G}_{N}^{\wedge k} 
				 } \left(f(\lambda_1,\dots,\lambda_k)\right)} = O\left(\frac{\log N}{N}\right) .
			\end{equation*}
		\end{enumerate} 
	\end{Thm}
	
	\medskip
	Before we get to the proof of the theorem we will prove a general Lemma.
	
	\begin{Lemma} \label{determinantal process expectation bound}
		Let $\mathcal{Y}_1,\mathcal{Y}_2 $ be determinantal point processes on $\mathbb{C}$ with kernels $K_1,K_2 $ with respect to the Lebesgue background measure respectively. Let $k $ be a natural number and $f \colon \mathbb{C}^k \to \mathbb{C}$ be a bounded Lipschitz function with support contained in $ S^k$ for a compact set $S\subseteq \mathbb{C} $.
		Assume that for all $x,y\in S $, $\abs{K_1(x,y)},\abs{K_2(x,y)} \leq A $ and $\abs{K_1(x,y)-K_2(x,y)}\leq B $.
		Denote for $m=1,2$, $E_{m,f} = \mathbb{E}(\sum f(y_1,\dots, y_k)) $ where the sum is over all the ordered subsets of $\mathcal{Y}_m $ of size $ k$. Then for some constant $C$ depending on $k,f $
		
		\begin{equation} \label{expectation of sum difference equation}
			\abs{E_{1,f}-E_{2,f}} \leq CA^{k-1}B .
		\end{equation}
		
		In addition, if we assume that each of those processes contains exactly $ N$ points, then 
		\begin{equation} \label{expectation of random subsets difference equation}
			\abs{\mathbb{E}_{\{\lambda_1,\dots,\lambda_k\} \sim \mathcal{Y}_{1}^{\wedge k} }f(\lambda_1,\dots,\lambda_k) - \mathbb{E}_{\{\lambda_1,\dots,\lambda_k\} \sim \mathcal{Y}_{2}^{\wedge k} }f(\lambda_1,\dots,\lambda_k)} \leq \frac{CA^{k-1}B}{N^{k}}.
		\end{equation}
	\end{Lemma}
	
	\begin{rmk}
		We do not care about the dependence on $ k$ so we show an elementary proof here.	It is possible to get a better constant $C$, using Lemma~3.4.2 of \cite{IntroRandomMatrices}. 
	\end{rmk}
	\begin{proof}
		For $m=1,2 $ the $k$-point correlation function of $\mathcal{Y}_i $ is 
		\begin{align*}
			p_{m}^{(k)}(x_1,\dots,x_k) = \det \left(K_m(x_i,x_j)\right)_{i,j=1}^{k} = \sum_{\pi \in S_{k}} \sgn{\pi} \prod_{i=1}^{k} K_m\left(x_i,x_{\pi(i)}\right).
		\end{align*}
		Hence, for $x_1,\dots x_k \in S $,
		\begin{align} \label{correlation function bound}
			\big| p_{1}^{(k)} (x_1,\dots,x_k) &-p_{2}^{(k)} (x_1,\dots,x_k) \big| \nonumber \\
			&= \abs{\sum_{\pi \in S_{k}} \sgn{\pi} \left(\prod_{i=1}^{k} K_1\left(x_i,x_{\pi(i)}\right)-\prod_{i=1}^{k} K_2\left(x_i,x_{\pi(i)}\right)\right)} \nonumber \\
			&\leq  \sum_{\pi \in S_{k}} \abs{\prod_{i=1}^{k} K_1\left(x_i,x_{\pi(i)}\right)-\prod_{i=1}^{k} K_2\left(x_i,x_{\pi(i)}\right)}. 
		\end{align}
		We denote $K_m\left(x_i,x_{\pi(i)}\right)=K_{m}^{\pi,i} $ and then 
		\begin{align} \label{product difference}
			\abs{\prod_{i=1}^{k} K_1^{\pi,i}-\prod_{i=1}^{k} K_2^{\pi,i}} &= \abs{\sum_{j=1}^{k} \left[\prod_{i=1}^{j-1}K_1^{\pi,i}\right] \left(K_1^{\pi,j}-K_2^{\pi,j}\right) \left[\prod_{i=j+1}^{k}K_2^{\pi,i}\right]} \nonumber \\
			&\leq k A^{k-1}B .
		\end{align}
		Substituting \eqref{product difference} into \eqref{correlation function bound} we get 
		\begin{equation} \label{partition function difference final bound}
			\abs{p_{1}^{(k)} (x_1,\dots,x_k) -p_{2}^{(k)} (x_1,\dots,x_k)} \leq \sum_{\pi \in S_k}  k A^{k-1}B = k! \cdot kA^{k-1}B .
		\end{equation}
		Hence 
		\begin{align*}
			\abs{E_{1,f}-E_{2,f}} &\leq \int_{\mathbb{C}^k} \abs{f(x_1,\dots,x_k)} \abs{p_{1}^{(k)} (x_1,\dots,x_k) -p_{2}^{(k)} (x_1,\dots,x_k)} \mathrm{d}x_1 \cdots \mathrm{d}x_k \\
			&\leq k! k A^{k-1} B \cdot \operatorname{Vol}\left(\operatorname{supp}(f)\right) \max \abs{f} .
		\end{align*}
		That proves \eqref{expectation of sum difference equation}.
		
		Now, assuming each process contains exactly $ N$ points, we have
		\begin{equation*}
			\mathbb{E}_{\{\lambda_1,\dots,\lambda_k\} \sim \mathcal{Y}_{m}^{\wedge k} }f(\lambda_1,\dots,\lambda_k) = \binom{N}{k}^{-1} \int_{\mathbb{C}^k} f(x_1,\dots,x_k) p_{m}^{(k)}(x_1,\dots,x_k) \mathrm{d}x_1 \cdots \mathrm{d}x_k .
		\end{equation*}
		Hence,
		\begin{align*}
			&\abs{\mathbb{E}_{\{\lambda_1,\dots,\lambda_k\} \sim \mathcal{Y}_{1}^{\wedge k} }f(\lambda_1,\dots,\lambda_k) - \mathbb{E}_{\{\lambda_1,\dots,\lambda_k\} \sim \mathcal{Y}_{2}^{\wedge k} }f(\lambda_1,\dots,\lambda_k)} \\
			&\leq \binom{N}{k}^{-1} k! k A^{k-1} B \cdot \operatorname{Vol}\left(\operatorname{supp}(f)\right) \max \abs{f}.
		\end{align*}
		Since we have $\binom{N}{k} = \prod_{i=1}^{k}\frac{N-i}{k-i} \geq \left(\frac{N}{k}\right)^k $ we conclude that for $ C(k,f)=C>0$ we have
		\begin{equation*}
			\abs{\mathbb{E}_{\{\lambda_1,\dots,\lambda_k\} \sim \mathcal{Y}_{1}^{\wedge k} }f(\lambda_1,\dots,\lambda_k) - \mathbb{E}_{\{\lambda_1,\dots,\lambda_k\} \sim \mathcal{Y}_{2}^{\wedge k} }f(\lambda_1,\dots,\lambda_k)} \leq \frac{C A^{k-1} B}{N^k} .
		\end{equation*}
		
	\end{proof}

	Now we can go back to prove the two parts of Proposition~\ref{Ginibre regions}.
	
	\begin{proof}[Proof of part 1. of Theorem~\ref{Ginibre regions}]
		As $\operatorname{supp}f \subseteq \{z\in \mathbb{C} \colon \abs{z}<R\}^k$, to calculate $E_f $ we can look at the process restricted to the disk $ \{z\in \mathbb{C} \colon \abs{z}<R\}$. According to Theorem~\ref{conditioned mixture inner}, this is a mixture of determinantal point processes indexed by subsets $ {J \subseteq \{0,1,\dots,N-1\}}$ of size $N_c$ with kernels 
		\begin{equation*}
			K_{\text{inner}}^{J}(z,w) = \sum_{k \notin J} \phi_k (z) \overline{\phi_k (w)} \mathds{1}_{\{\abs{z},\abs{w}<R\}} .
		\end{equation*}
		
		According to Lemma~\ref{bound on big sum} with probability $1 - O(N^{-10})$ we have that $\sum_k \left(n(J)_k\right) <M\log(N) $ and in particular every $k=1,\dots, N_c$ satisfies $ n_k < M\log(N) $ and moreover for $k> M\log(N) $, $n_k=0$. Now, we denote $l_k = N-N_c+k$ and by \eqref{inner orthogonal pol def} we have, for all $x,y\in D(0,R) $, $k \leq M \log N$,
		\begin{align*} 
			\delta_k &= \abs{\tilde{\phi}_{l_k-n_k}(x) \overline{\tilde{\phi}_{l_k-n_k}(y)} - \tilde{\phi}_{l_k}(x) \overline{\tilde{\phi}_{l_k}(y)}} \\
			&=\abs{\frac{N^{l_k-n_k+1} \abs{xy}^{l_k-n_k}}{\pi \gamma(l_k-n_k+1,NR^2)} - \frac{N^{l_k+1} \abs{xy}^{l_k}}{\pi \gamma(l_k+1,NR^2)}}e^{-\frac{N}{2}(\abs{x}^2+\abs{y}^2)} \\
			&\leq  N^{l_k-n_k+1}\abs{xy}^{l_k-n_k}e^{-\frac{N}{2}(\abs{x}^2+\abs{y}^2)} \abs{\frac{1}{\gamma(l_k-n_k+1,NR^2)} - \frac{N^{n_k}\abs{xy}^{n_k}}{\gamma(l_k+1,NR^2)}}. \\
		\end{align*}
		By the triangle inequality we get 
		\begin{align} \label{delta k estimate}
			\delta_k &\leq N^{l_k-n_k+1}\abs{xy}^{l_k-n_k}e^{-\frac{N}{2}(\abs{x}^2+\abs{y}^2)} \bigg(\abs{\frac{1}{\Gamma(l_k-n_k+1)} -\frac{N^{n_k}\abs{xy}^{n_k}}{\Gamma(l_k+1)} } \nonumber \\
			&+ \abs{\frac{1}{\Gamma(l_k-n_k+1)} -\frac{1}{\gamma(l_k-n_k+1,NR^2)}}  + \abs{\frac{N^{n_k}\abs{xy}^{n_k}}{\Gamma(l_k+1)} - \frac{N^{n_k}\abs{xy}^{n_k}}{\gamma(l_k+1,NR^2)}} \bigg) \nonumber \\
			&= A_1 + N^{l_k-n_k+1}\abs{xy}^{l_k-n_k}e^{-\frac{N}{2}(\abs{x}^2+\abs{y}^2)}\left(A_2 + A_3\right).
		\end{align}
		where we denoted 
		\begin{align*}
			A_1 &= N^{l_k-n_k+1}\abs{xy}^{l_k-n_k}e^{-\frac{N}{2}(\abs{x}^2+\abs{y}^2)} \abs{\frac{1}{\Gamma(l_k-n_k+1)}-\frac{N^{n_k}\abs{xy}^{n_k}}{\Gamma(l_k+1)}} \\
			A_2 &= \abs{\frac{1}{\Gamma(l_k-n_k+1)} -\frac{1}{\gamma(l_k-n_k+1,NR^2)}} \\
			A_3 &=  \abs{\frac{N^{n_k}\abs{xy}^{n_k}}{\Gamma(l_k+1)} - \frac{N^{n_k}\abs{xy}^{n_k}}{\gamma(l_k+1,NR^2)}}.
		\end{align*}
		Then we have by Lemma~\ref{Incomplete Gamma estimates}
		\begin{align} 
			A_2 &= \frac{ \abs{\gamma(l_k-n_k+1,NR^2)-\Gamma(l_k-n_k+1)}} {\gamma(l_k-n_k+1,NR^2)\Gamma(l_k-n_k+1)}  \\ 
			&= \frac{\Gamma(l_k-n_k+1)-\gamma(l_k-n_k+1,NR^2)}{\gamma(l_k-n_k+1,NR^2)} \nonumber \\
			&\leq K e^{-NR^2+(l_k-n_k+1)(1+\log\frac{NR^2}{l_k-n_k+1})} \nonumber 
		\end{align}
		for some constant $K>0$. Hence, as $l_k - N(1-c),n_k =O(\log N) $,
		\begin{equation} \label{A2 estimate}
			A_2 \leq  e^{N\left(-R^2+(1-c)(1+\log \frac{R^2}{1-c})\right) (1+O(\log(N))) }  .
		\end{equation}
		
		Similarly 
		\begin{equation} \label{A3 estimate}
			A_3 \leq  N^{n_k}\abs{xy}^{n_k}e^{N\left(-R^2+(1-c)(1+\log \frac{R^2}{1-c})\right) (1+O(\log(N))) }  .
		\end{equation}
		
		Meanwhile we have 
		\begin{align*}
			A_1 &= N^{l_k-n_k+1}\abs{xy}^{l_k-n_k}e^{-\frac{N}{2}(\abs{x}^2+\abs{y}^2)} \abs{\frac{1}{\Gamma(l_k-n_k+1)}-\frac{N^{n_k}\abs{xy}^{n_k}}{\Gamma(l_k+1)}} \\
			&= \frac{N^{l_k-n_k+1}\abs{xy}^{l_k-n_k}e^{-\frac{N}{2}(\abs{x}^2+\abs{y}^2)}}{ \Gamma(l_k+1)} \abs{\frac{\Gamma(l_k+1)}{\Gamma(l_k-n_k+1)}-N^{n_k}\abs{xy}^{n_k}}  .
		\end{align*}
		We will denote 
		\begin{equation*}
			g(s,t)= \left(st\right)^{\frac{1}{2}(l_k - n_k)} e^{-\frac{N}{2}(s+t)} \left(\frac{1}{\Gamma(l_k-n_k+1)}- \frac{N^{n_k}\left(st\right)^{\frac{1}{2}n_k}}{\Gamma(l_k+1)}\right)
		\end{equation*}
		so that $A_1 = N^{l_k-n_k+1} \abs{g\left(\abs{x}^2,\abs{y}^2\right)}$. In order to bound $A_1$ we will look for extrema of $g(s,t)$ in the quadrant $s,t\geq0 $.
		
		\begin{Lemma} \label{g maximum}
			For $g $ as above, 
			\begin{equation*}
				\max_{s,t>0} \abs{N^{l_k-n_k}g(s,t)} \leq \frac{n_k }{N(1-c)} \left(1+O\left(\frac{n_k}{\sqrt{l_k}}\right)\right).
			\end{equation*}
			
		\end{Lemma}
		\begin{proof}[Proof of Lemma~\ref{g maximum}]

			Note that as $g(0,t)=g(\infty,t)=g(s,0)=g(s,\infty)=0 $ we can look only at points where $ \nabla g(s,t) = (0,0)$. We compute 
			\begin{align*}
				0 &= \frac{\partial g}{\partial s} = g(s,t)\left[\frac{l_k - n_k}{2s} -\frac{N}{2} - \frac{\frac{N^{n_k}n_k\left(st\right)^{\frac{1}{2}n_k}}{2s \Gamma(l_k+1)}}{\frac{1}{\Gamma(l_k-n_k+1)}- \frac{N^{n_k}\left(st\right)^{\frac{1}{2}n_k}}{\Gamma(l_k+1)}} \right] , \\
				0 &= \left(\frac{l_k-n_k}{s}-N\right) \left(\frac{1}{\Gamma(l_k-n_k+1)}- \frac{N^{n_k}\left(st\right)^{\frac{1}{2}n_k}}{\Gamma(l_k+1)} \right) -\frac{N^{n_k}n_k\left(st\right)^{\frac{1}{2}n_k}}{s \Gamma(l_k+1)} .
			\end{align*}
			Rearranging the terms we get
			\begin{align*}
				n_k N^{n_k} \left(st\right)^{\frac{1}{2}n_k} &= \left(l_k-n_k-Ns\right)\left(\prod_{i=0}^{n_k-1} (l_k-i) - N^{n_k} \left(st\right)^{\frac{1}{2}n_k}\right) , \\
				l_k - n_k -Ns &=\frac{n_k N^{n_k} \left(st\right)^{\frac{1}{2}n_k}}{\prod_{i=0}^{n_k-1} (l_k-i) - N^{n_k} \left(st\right)^{\frac{1}{2}n_k}} .
			\end{align*}
		
			Similarly, the equation $ \frac{\partial g}{\partial t} =0$  gives $ l_k - n_k -Nt =\frac{n_k N^{n_k} \left(st\right)^{\frac{1}{2}n_k}}{\prod_{i=0}^{n_k-1} (l_k-i) - N^{n_k} \left(st\right)^{\frac{1}{2}n_k}} $ and hence the extremal points satisfy $s=t$. Substituting it back we get 
			\begin{equation} \label{s equation}
				n_k N^{n_k} s^{n_k} = \left(l_k-n_k-Ns\right)\left(\prod_{i=0}^{n_k-1} (l_k-i) - N^{n_k} s^{n_k}\right) .
			\end{equation}
			
			We denote $ \tilde{l}_k = \sqrt[n_k]{\prod_{i=0}^{n_k-1} (l_k-i)} $ and rewrite the equation as
			
			\begin{align*}
				\left(\frac{l_k-n_k}{N} -s \right)\left(\left(\frac{\tilde{l}_k}{Ns}\right)^{n_k} - 1\right) = \frac{n_k}{N} .
			\end{align*}
			
			As $ \frac{\tilde{l}_k}{N} = \frac{l_k -n_k}{N} +O(\frac{n_k}{N}) $ the solutions to that equation must satisfy $ s = \frac{\tilde{l_k}}{N} +O(\sqrt{\frac{n_k}{N}}) $, otherwise each of the terms in the LHS is bigger in absolute value than the square root of the RHS. So we can write $ s= \frac{l_k}{N}(1+\epsilon) $ with $ \epsilon = O(\sqrt{\frac{n_k}{N}}) $  and so 
			\begin{align*}
				\left(-\epsilon \frac{l_k}{N} + O\left(\frac{n_k}{N}\right)\right) &\left(-n_k\epsilon +O\left(\frac{n_k^2}{N} + \epsilon^2 n_k^2 \right)\right) = \frac{n_k}{N} \\
				\epsilon &= \pm \frac{1}{\sqrt{l_k}} + O\left(\frac{n_k}{N}\right)
			\end{align*}
		
			We conclude that the maximum of $\abs{g(s,t)}$ is attained at $ s_{\max}=t_{\max}= \frac{l_k}{N}(1+\epsilon) = \frac{l_k \pm \sqrt{l_k} +O(n_k)}{N}  $. 
			Thus
			\begin{align*}
				\max_{s,t>0} \abs{N^{l_k-n_k}g(s,t)} &= e^{(l_k-n_k)\log \left(Ns_{\max}\right)-Ns_{\max}} \abs{\frac{1}{\Gamma(l_k-n_k+1)}- \frac{\left(Ns_{\max}\right)^{n_k}}{\Gamma(l_k+1)}} \\
				&= \frac{e^{\phi(-Ns_{\max})}}{\Gamma(l_k+1)} \abs{\prod_{i=0}^{n_k-1} (l_k-i) - \left(l_k \pm \sqrt{l_k} +O(n_k)\right)^{n_k}}
			\end{align*}
			where we denoted $\phi(x) = -x+ (l_k-n_k)\log x$. Since $\sqrt{l_k} \gg n_k $ we get
			\begin{align} \label{g maximum estimate}
				\max_{s,t>0} \abs{N^{l_k-n_k}g(s,t)} &= \frac{e^{\phi(Ns_{\max})}}{ \Gamma(l_k+1)} \abs{\pm n_k l_k^{n_k-1} \sqrt{l_k} \left(1+O\left(\frac{n_k}{\sqrt{l_k}}\right)\right) } \nonumber \\
				&\leq \frac{n_k l_k^{n_k-\frac{1}{2}} e^{\phi((l_k-n_k) \pm \sqrt{l_k} +O(n_k)) }}{\Gamma(l_k+1)} \left(1+O\left(\frac{n_k}{\sqrt{l_k}}\right)\right).
			\end{align}		
			A Taylor expansion of $ \phi $ around $l_k -n_k $ gives 
			\begin{equation*}
				\phi(l_k-n_k+y) = \phi(l_k-n_k) +\frac{y^2}{2(l_k-n_k)} +O\left(\frac{y^3}{l_k^2}\right) 
			\end{equation*} 
			and substituting it into \eqref{g maximum estimate} we get
			\begin{align*}
				\max_{s,t>0} \abs{N^{l_k-n_k}g(s,t)} &\leq \frac{n_k l_k^{n_k-\frac{1}{2}} e^{\phi(l_k-n_k) + \frac{1}{2(l_k-n_k)} (l_k+O(n_k\sqrt{l_k})) + O(\nicefrac{1}{\sqrt{l_k}}) }}{\Gamma(l_k+1)} \left(1+O\left(\frac{n_k}{\sqrt{l_k}}\right)\right) \\
				&= \frac{n_k l_k^{n_k-\frac{1}{2}}  e^{-l_k+n_k +(l_k-n_k)\log (l_k-n_k)}  e^{ \frac{1}{2}}}{\Gamma(l_k+1)} \left(1+O\left(\frac{n_k}{\sqrt{l_k}}\right)\right) \\
				&= \frac{n_k  e^{-l_k+n_k +(l_k-n_k)\log (l_k-n_k)+n_k\log l_k}  e^{ \frac{1}{2}}}{\sqrt{l_k}\Gamma(l_k+1)} \left(1+O\left(\frac{n_k}{\sqrt{l_k}}\right)\right) .
			\end{align*}
			Using Stirling`s approximation we get
			\begin{align*}
				\max_{s,t>0} \abs{N^{l_k-n_k}g(s,t)} &\leq \frac{n_k  e^{-l_k+n_k +(l_k-n_k)\log (l_k-n_k)+n_k\log l_k}  e^{ \frac{1}{2}}}{\sqrt{2\pi}l_k \left(\frac{l_k}{e}\right)^{l_k}} \left(1+O\left(\frac{n_k}{\sqrt{l_k}}\right)\right) \\
				&\leq \frac{ n_k  e^{n_k +(l_k-n_k)\log (\frac{l_k-n_k}{l_k})}  }{l_k } \left(1+O\left(\frac{n_k}{\sqrt{l_k}}\right)\right) \\
				&= \frac{ n_k  e^{n_k +(l_k-n_k) \left(-\frac{n_k}{l_k} + O\left(\frac{n_k^2}{l_k^2}\right) \right)}  }{l_k } 	\left(1+O\left(\frac{n_k}{\sqrt{l_k}}\right)\right) \\
				&= \frac{n_k }{N(1-c)} \left(1+O\left(\frac{n_k}{\sqrt{l_k}}\right)\right) .
			\end{align*}
		\end{proof}
		 
		Using Lemma~\ref{g maximum} we conclude that
		\begin{equation} \label{A1 estimate}
			A_1 \leq N\max \abs{N^{l_k-n_k}g(s,t)} \leq \frac{n_k }{1-c} \left(1+O\left(\frac{n_k}{\sqrt{l_k}}\right)\right).
		\end{equation}

		Now, combining \eqref{delta k estimate}, \eqref{A2 estimate},\eqref{A3 estimate} and \eqref{A1 estimate} we get $\delta_k \leq C \left(n_k + e^{-\epsilon N}\right)$, for some constants $C,\epsilon> 0 $. Hence we have 
		\begin{equation*}
			\abs{K_{\operatorname{inner}}^{J}(z,w)-K_{\operatorname{inner}}^{J_{\mathbf{0}}}(z,w) } < C M \log(N) .
		\end{equation*}
		
		Now, we note that the determinantal process with kernel $K_{\operatorname{inner}}^{J_{\mathbf{0}}} $ is the Ginibre ensemble of size $N-N_c $ multiplied by $\sqrt{\frac{N-N_c}{N}} $ conditioned on the event that no eigenvalues that are bigger in absolute value than $ R \sqrt{\frac{N}{N-N_c}} >1 $. The complementary event, that there exists an eigenvalue bigger in absolute value than $ R \sqrt{\frac{N}{N-N_c}}$ happens with exponentially small probability in $N $ \cite[Section 4.1]{GinibreLargeEVBound}, so we have
		
		\begin{equation*}
			\abs{K_{\operatorname{inner}}^{J}(z,w)-K_{\sqrt{\frac{N-N_c}{N}}\mathcal{G}_{N-N_c}}(z,w) } < \tilde{C} \log(N) .
		\end{equation*} 
		
		Hence, by Lemma~\ref{determinantal process expectation bound},  $\abs{E_f - \mathbb{E}^{\sqrt{\frac{N-N_c}{N}}\mathcal{G}_{N-N_c}}(f(\lambda_1,\dots, \lambda_k)) }  = O\left(\frac{\log N}{N}\right) $.

		\end{proof}
				
		\begin{proof}[Proof of part 2. of Theorem~\ref{Ginibre regions}]
		When we look at $f $ with support in $ \{ z\in \mathbb{C} | \abs{z}>R+\epsilon\}^k $, to calculate $E_f $ we need to consider only the distribution of eigenvalues outside the disk $D(0,R) $. According to Theorem~\ref{conditioned mixture} this is a mixture of determinantal point processes indexed by $ J \subseteq \{0,1,\dots,N-1\} $ of size $N_c $ with kernels $K^J(z,w)=\sum_{k\in J} \phi_{k}(z)\overline{\phi_k(w)} \mathds{1}_{\{\abs{z},\abs{w}>R\}}  $. Now, by Lemma~\ref{bound on big sum}, with probability at least $ 1-O(N^{-10}) $ the set $J $ is such that $\sum_k \left(n(J)_k\right) <M\log(N) $ for some constant $M>0 $, and by a similar computation to the one above we have $\abs{K^J(z,w)-K^{J_0}(z,w)}<C \log(N) $ for some constant $C>0 $. For $ \abs{z},\abs{w}>R+\epsilon $,  $\abs{K^{J_0}(z,w) - K^{\mathcal{G}_N}(z,w)} $ is exponentially small in $N $. And thus we conclude, by Lemma~\ref{determinantal process expectation bound},
		\begin{equation*}
			\abs{E_f - \mathbb{E}^{\mathcal{G}_N}(f(\lambda_1,\dots,\lambda_k))} = O\left(\frac{\log N}{N}\right) .
		\end{equation*}
	\end{proof}
	
	Now we will show convergence in the region $\abs{z}=R+O\left(\frac{1}{N}\right) $. Roughly speaking, for  $\epsilon>0  $ tending to $0$ slowly with $N$, the region ${\abs{z} > R+\epsilon} $ typically contains about $(1-R^2)N $ eigenvalues, and thus we expect the region ${\abs{z}=R+O\left(\frac{1}{N}\right)} $ to have about $(R^2-1+c)N $ eigenvalues. The circumference of the disk $D(0,R) $ is proportional to $R $ so we expect to get a normalised limit process when we zoom in by a factor of $ N\frac{R^2-1+c}{R} $. This is indeed what we find in the next theorem.
	
	\begin{Thm} \label{hard wall limit distribution}
		For any natural number $ k $, and any bounded Lipschitz function with compact support in $\mathbb{C}_+^k $, $f \colon \mathbb{C}^k \to \mathbb{C} $ , denote 
		\begin{equation*}
			 h_N(\lambda_1,\dots,\lambda_k) = f\left(\frac{N(R^2-1+c)}{R}(\lambda_1-R),\dots, \frac{N(R^2-1+c)}{R}(\lambda_k-R)\right) .
		\end{equation*}
		Then, denoting $E_{f,N}= \mathbb{E} \left(\sum h_N(\lambda_1,\dots,\lambda_k)\right)$ where the sum is over (ordered) subsets of $k $ distinct eigenvalues of the conditioned process and $E_{\mathcal{X},f} =\mathbb{E} \left(\sum f(x_1,\dots,x_k)\right) $ where the sum is over all the (ordered) subsets of $\mathcal{X} $ of size $k$, we have
		\begin{equation*}
			E_{f,N} = E_{\mathcal{X},f} + O\left(\frac{\log^2(N)}{N}\right) .
		\end{equation*}
		Above $\mathcal{X}$ is the determinantal process on the right half plane defined by the kernel 
		\begin{equation*}
			K_{\mathcal{X}}(z,w)= \frac{1-(z+\overline{w}+1)e^{-z-\overline{w}}}{\pi \left(z+\overline{w}\right)^2},
		\end{equation*}
		with respect to the Lebesgue background measure.
	\end{Thm}
	
	\begin{rmk}
		This kernel is positive definite and thus defines a determinantal point process since for any complex continuous function $ g $ with compact support in the right half plane we have
		\begin{align*}
			\iint \mathrm{d}z \mathrm{d}w K_{\mathcal{X}}(z,w) g(z)\overline{g(w)} &= \frac{1}{\pi}\iint \mathrm{d}z \mathrm{d}w \int_{0}^{1}   se^{-s(z+\overline{w})}  g(z)\overline{g(w)} \mathrm{d}s \\
			&=\frac{1}{\pi} \int_{0}^{1}  s \abs{\int \mathrm{d} z e^{-sz}g(z)}^2 \mathrm{d}s \geq 0 .
		\end{align*}
	\end{rmk}
	
	\begin{proof}[Proof of Theorem~\ref{hard wall limit distribution}]
		For big enough $N$, $h_N(\lambda_1,\dots,\lambda_k) $ is non zero only for eigenvalues outside of the disk $D(0,R) $ and as such to calculate $E_{f,N} $ we will look only at their distribution. According to Theorem~\ref{conditioned mixture} this is a mixture of determinantal point processes indexed by $ J \subseteq \{0,1,\dots,N-1\} $ of size $N_c $ with kernels $K^J(z,w)=\sum_{k\in J} \phi_{k}(z)\overline{\phi_k(w)} \mathds{1}_{\{\abs{z},\abs{w}>R\}}  $.
		
		We want to find the distribution of the variables $x_i = \frac{N(R^2-1+c)}{R}(\lambda_i-R) $. For ease of computation we will first look at the distribution of $\mathcal{Y}=\{y_1,\dots, y_N\} $ where $y_i = N(\lambda_i-R) $ . They are distributed as a mixture of determinantal point processes with the same probabilities and their kernels are
		\begin{align*}
			\tilde{K}^J_{\mathcal{Y}}(z,w) &= \frac{1}{N^2}K^J\left(R+\frac{1}{N}z,R+\frac{1}{N}w\right) \\
			&= \frac{1}{N^2} \sum_{k \in J} \frac{N^{k+1} e^{-\frac{1}{2}N \left(\abs{R+\frac{1}{N}z}^2+\abs{R+\frac{1}{N}w}^2\right)}}{\pi \Gamma(k+1,NR^2)} \left((R+\frac{1}{N}z)(R+\frac{1}{N}\overline{w})\right)^k \\
			&= \sum_{k\in J}\frac{N^{k-1} e^{-NR^2}R^{2k}}{\pi \Gamma(k+1,NR^2)} e^{-\frac{1}{2}R(z+\overline{z}+w+\overline{w})-\frac{\abs{z}^2+\abs{w}^2}{2N}} \left(1+\frac{z+\overline{w}}{NR}+\frac{z\overline{w}}{N^2R^2}\right)^k .
		\end{align*}
		The function $f $ is compactly supported so there exists a compact set $K_{\operatorname{supp}} \subseteq \mathbb{C} $ such that $\operatorname{supp}(f) \subseteq K_{\operatorname{supp}}^k  $. Thus it is enough to estimate the values of $\tilde{K}^{J}(z,w) $ for $ z,w \in K_{\operatorname{supp}} $.
		
		By Lemma~\ref{bound on big sum}, with probability at least $ 1-O(N^{-10}) $ the set $J $ is such that $\sum_k \left(n(J)_k\right) <M\log(N) $ for some constant $M>0 $. Hence for such values of $J $, for each $ k\in J\backslash J_0 \cup J_0 \backslash J $ we have $ \abs{k-(N-N_c)} \leq M\log(N) $. Hence, by Lemma~\ref{Incomplete Gamma estimates}
		\begin{align*}
			\frac{N^{k-1} e^{-NR^2}R^{2k}}{\pi \Gamma(k+1,NR^2)} &= \frac{\sqrt{2} e^{-(k+1)\log\frac{NR^2}{k+1}-k-1} \left(\frac{NR^2}{k+1}-1\right)\sqrt{k+1} N^{k-1} R^{2k}}{\sqrt{\pi} \Gamma(k+1)} \left(1+O\left(\frac{1}{N}\right)\right) \\
			&= \frac{ e^{(k+1)\log(k+1)-k\log(k)} (NR^2)^{-k-1} \left(\frac{NR^2}{k+1}-1\right)}{e\pi }   \\
			&\cdot \sqrt{1+\frac{1}{k}} N^{k-1} R^{2k} \left(1+O\left(\frac{1}{N}\right)\right) \\
			&= \frac{ (k+1)e^{k\log\left(1+\frac{1}{k}\right)} \sqrt{1+\frac{1}{k}} }{e\pi N^2 R^2 } \left(1+O\left(\frac{1}{N}\right)\right) = O\left(\frac{1}{N}\right) .
		\end{align*}
		Thus we conclude that for some constant $ \tilde{M}>0$, for all such values of $J$
		\begin{equation} \label{edge kernels similarity}
			\abs{\tilde{K}^{J_0}_{\mathcal{Y}}(z,w)-\tilde{K}^{J}_{\mathcal{Y}}(z,w)} \leq \tilde{M}\frac{\log(N)}{N} .
		\end{equation}
		
		Now we will estimate
		\begin{equation*}
			\tilde{K}_{\mathcal{Y}}^{J_0} = \sum_{n=N-N_c+1}^{N} \frac{N^{n-1} e^{-NR^2}R^{2n}}{\pi \Gamma(n+1,NR^2)} e^{-\frac{1}{2}R(z+\overline{z}+w+\overline{w})-\frac{\abs{z}^2+\abs{w}^2}{2N}} \left(1+\frac{z+\overline{w}}{NR}+\frac{z\overline{w}}{N^2R^2}\right)^n  .
		\end{equation*}
		We denote $\delta_N = \sqrt{N}\log(N) $ and then denote by $E_1, E_2, E_3 $ the sums from $n=N-N_c+1 $ to $ \lfloor NR^2 -\delta_N \rfloor$, from $n= \lfloor NR^2 -\delta_N \rfloor +1$ to $ \lfloor NR^2+\delta_N \rfloor$ and from $n= \lfloor NR^2+\delta_N\rfloor+1$ to $N$ respectively. Then we have
		\begin{equation} \label{edge kernel decomposition}
			\tilde{K}_{\mathcal{Y}}^{J_0}(z,w) = E_{1} + E_{2} + E_{3} .
		\end{equation}
		First we will estimate $ E_{1} $.
		\begin{align*}
			E_{1} &= \sum_{n=N-N_c+1}^{\lfloor NR^2 -\delta_N \rfloor} \frac{N^{n-1} e^{-NR^2}R^{2n}}{\pi \Gamma(n+1,NR^2)} e^{-\frac{1}{2}R(z+\overline{z}+w+\overline{w})-\frac{\abs{z}^2+\abs{w}^2}{2N}} \left(1+\frac{z+\overline{w}}{NR}+\frac{z\overline{w}}{N^2R^2}\right)^n \\
			&=  \sum_{n=N-N_c+1}^{\lfloor NR^2 -\delta_N \rfloor} \frac{N^{n-1} e^{-NR^2}R^{2n}}{\pi \Gamma(n+1,NR^2)} e^{-R\cdot \Re(z+w)} \left(1+\frac{z+\overline{w}}{NR}\right)^n \left(1+O\left(\frac{1}{N}\right)\right) .
		\end{align*}
		For $ N -N_c +1 \leq n \leq NR^2 - \delta_N $ we denote $\tau_n = \frac{n}{N} $ and then we have, by Lemma~\ref{Incomplete Gamma estimates},
		\begin{align*}
			\Gamma(n+1,NR^2) &= \Gamma(n+1) \frac{e^{-NR^2+(n+1)\left(1+\log \frac{NR^2}{n+1}\right)}} {\sqrt{2\pi(n+1)}\left(\frac{NR^2}{n+1}-1\right)}  \left(1+O\left(\frac{1}{n\left(\frac{NR^2}{n}-1\right)^2}\right)\right) \\
			&= \frac{\sqrt{2\pi (n+1)} e^{(n+1)(\log(n+1)-1)}}{n+1} \cdot \\
			&\cdot \frac{e^{-NR^2+(n+1)\left(1+\log \frac{NR^2}{n+1}\right)}} {\sqrt{2\pi(n+1)}\left(\frac{NR^2}{n+1}-1\right)}  \left(1+O\left(\frac{1}{n\left(\frac{NR^2}{n}-1\right)^2}\right)\right) \\
			&= \frac{e^{-NR^2} (NR^2)^{n+1}} {n\left(\frac{NR^2}{n}-1\right)}  \left(1+O\left(\frac{1}{N\left(\frac{NR^2}{n}-1\right)^2}\right)\right)
		\end{align*}
		and
		\begin{equation*}
			\frac{N^{n-1}e^{-NR^2}R^{2N}}{\Gamma(n+1,NR^2)} = \frac{R^2-\tau_n}{NR^2} \left(1+O\left(\frac{1}{N\left(\frac{NR^2}{n}-1\right)^2}\right)\right) .
		\end{equation*}
		Thus we conclude 
		\begin{align*}
			E_{1} &=  \sum_{n=N-N_c+1}^{\lfloor NR^2 -\delta_N \rfloor} \frac{N^{n-1} e^{-NR^2}R^{2n}}{\pi \Gamma(n+1,NR^2)} e^{-R\cdot \Re(z+w)} \left(1+\frac{z+\overline{w}}{NR}\right)^n \left(1+O\left(\frac{1}{N}\right)\right) \\
			&= \sum_{n=N-N_c+1}^{\lfloor NR^2 -\delta_N \rfloor} \frac{R^2-\tau_n}{\pi NR^2} e^{-R\cdot \Re(z+w)+\tau_n \frac{z+\overline{w}}{R} }  \left(1+O\left(\frac{1}{N\left(\frac{NR^2}{n}-1\right)^2}\right)\right) \\
			&= \frac{e^{-R\cdot \Re(z+w)}}{\pi N R^2} \sum_{n=N-N_c+1}^{\lfloor NR^2 -\delta_N \rfloor} \left(R^2-\tau_n\right) e^{\tau_n \frac{z+\overline{w}}{R} }  \left(1+O\left(\frac{1}{N\left(\frac{NR^2}{n}-1\right)^2}\right)\right) .
		\end{align*}
		We notice this is a Riemann sum and thus 
		\begin{align*}
			\frac{1}{N} &\sum_{n=N-N_c+1}^{\lfloor NR^2 -\delta_N \rfloor} \left(R^2-\tau_n\right) e^{\tau_n \frac{z+\overline{w}}{R} }  \left(1+O\left(\frac{1}{N\left(\frac{NR^2}{n}-1\right)^2}\right)\right) \\ 
			&=  \int_{1-c}^{R^2} \left(R^2-t\right) e^{t\frac{z+\overline{w}}{R}} \mathrm{d}t +O\left(\frac{\delta_N}{N} + \int_{1-c}^{R^2 - \frac{\delta_N^2}{N^2}} \frac{\mathrm{d}t}{N(t-R^2)} \right) .
		\end{align*}
		Substituting $ s=\frac{R^2-t}{R}(z+\overline{w}) $ we get
		\begin{align} \label{E1 estimate}
			E_{1} &= \frac{e^{-R\cdot \Re(z+w)}}{\pi \left(z+\overline{w}\right)^2} \int_{0}^{\frac{R^2-1+c}{R}\left(z+\overline{w}\right)} se^{-s} \mathrm{d}s + O\left( \frac{\log^2 N}{N}\right) \nonumber \\
			&= \frac{e^{-R\cdot \Re(z+w)}}{\pi \left(z+\overline{w}\right)^2} \left[1 - \left(\frac{R^2-1+c}{R}(z+\overline{w})+1\right) e^{-\frac{R^2-1+c}{R}\left(z+\overline{w}\right)}\right] + O\left( \frac{\log^2 N}{N}\right) .
		\end{align}

		Now, to bound $E_{2} $ and $E_{3} $  we will notice that, by Lemma~\ref{series rep of incomplete} ,
		\begin{equation} \label{summand representation E2}
			\frac{N^{n-1} e^{-NR^2}R^{2n}}{\pi \Gamma(n+1,NR^2)} = \frac{1}{\pi N} \frac{\frac{\left(NR^2\right)^n}{n!}}{\sum_{k=0}^{n} \frac{\left(NR^2\right)^k}{k!}} .
		\end{equation}
		We denote 
		\begin{equation} \label{kappa n definition}
			k_n = \begin{cases}
				\frac{NR^2}{NR^2-n} & n<NR^2 -\sqrt{NR^2} , \\
				\sqrt{NR^2} & n\geq NR^2-\sqrt{NR^2} ,
			\end{cases}
		\end{equation}
		and then
		\begin{align*}
			\frac{\frac{\left(NR^2\right)^n}{n!}}{\sum_{k=0}^{n} \frac{\left(NR^2\right)^k}{k!}} &=  \left(\sum_{k=0}^{n} \prod_{l=0}^{k-1} \frac{n-l}{NR^2} \right)^{-1} \\
			&=  \left(\sum_{k=0}^{n} \exp \left[\sum_{l=0}^{k-1} \log \frac{n-l}{NR^2}\right]  \right)^{-1} \\
			&\leq \left(\sum_{k=0}^{\lfloor k_n \rfloor} \exp \left[\sum_{l=0}^{k-1} \log \left(1-\frac{NR^2-n+l}{NR^2}\right)\right]  \right)^{-1}  .
		\end{align*}
		Using the fact that for $ x \in [0,0.5]$, $\log(1-x)\geq -2x$ we get, for $N $ big enough such that ${0.5NR^2 > \sqrt{NR^2} +\delta_N} $ ,
		\begin{align*}
			\frac{\frac{\left(NR^2\right)^n}{n!}}{\sum_{k=0}^{n} \frac{\left(NR^2\right)^k}{k!}} &\leq  \left(\sum_{k=0}^{\lfloor k_n \rfloor} \exp \left[\sum_{l=0}^{k-1} -2 \max\left(\frac{NR^2-n+l}{NR^2},0\right)\right]  \right)^{-1} \\
			&\leq \left(\lfloor k_n \rfloor \exp \left[-2 \sum_{l=0}^{\lfloor k_n \rfloor}\max\left(\frac{NR^2-n+l}{NR^2},0\right) \right] \right)^{-1}  \\
		\end{align*}
		and substituting \eqref{kappa n definition} we have
		\begin{align} \label{summand estimate E2}
			\frac{\frac{\left(NR^2\right)^n}{n!}}{\sum_{k=0}^{n} \frac{\left(NR^2\right)^k}{k!}} &\leq \left(\lfloor k_n \rfloor \exp \left[-2 \sum_{l=0}^{\lfloor k_n \rfloor}\max\left(\frac{NR^2-n+l}{NR^2},0\right) \right] \right)^{-1} \nonumber \\
			&\leq \left(\lfloor k_n \rfloor e^{-4} \right)^{-1} \leq \frac{e^4}{k_n-1} \leq \frac{C}{\sqrt{N}}
		\end{align}
		for some constant $C>0 $. Combining \eqref{summand representation E2} and \eqref{summand estimate E2} we conclude that for a different constant $\tilde{C}>0 $,
		\begin{align} \label{E2 estimate}
			E_2 &=  \sum_{n=\lfloor NR^2-\delta_N\rfloor+1}^{\lfloor NR^2+\delta_N \rfloor} \frac{N^{n-1} e^{-NR^2}R^{2n}}{\pi \Gamma(n+1,NR^2)} e^{-\frac{1}{2}R(z+\overline{z}+w+\overline{w})-\frac{\abs{z}^2+\abs{w}^2}{2N}} \left(1+\frac{z+\overline{w}}{NR}+\frac{z\overline{w}}{N^2R^2}\right)^n \nonumber \\
			&\leq \sum_{n=\lfloor NR^2-\delta_N\rfloor+1}^{\lfloor NR^2+\delta_N \rfloor} \frac{\tilde{C}}{N\sqrt{N}} \leq \frac{2\tilde{C}\delta_N}{N\sqrt{N}} = O\left(\frac{\log N}{N}\right).
		\end{align}

		
		For $ n > \lfloor NR^2+\delta_N \rfloor $ we have, by Lemma~\ref{Incomplete Gamma estimates}, $\Gamma(n+1,NR^2) = n! \cdot (1+o(1)) $  and hence, while denoting $n_+ = \lfloor NR^2+\delta_N\rfloor+1 $, we have for some other constants $ C, \tilde{C}>0 $ depending on $K_{\operatorname{supp}} $,
		\begin{align} \label{E3 estimate}
			E_{3} &= \sum_{n=\lfloor NR^2+\delta_N\rfloor+1}^{N} \frac{N^{n-1} e^{-NR^2}R^{2n}}{\pi \Gamma(n+1,NR^2)} e^{-\frac{1}{2}R(z+\overline{z}+w+\overline{w})-\frac{\abs{z}^2+\abs{w}^2}{2N}} \left(1+\frac{z+\overline{w}}{NR}+\frac{z\overline{w}}{N^2R^2}\right)^n \nonumber \\
			&\leq C \max_{n=n_{+},\dots,N} \frac{(NR^2)^n e^{-NR^2}}{n!} = C\frac{(NR^2)^{n_+} e^{-NR^2}}{n_{+}!} \leq \tilde{C} \frac{\left(NR^2\right)^{n_+} e^{n_{+}-NR^2}}{n_{+}^{n_{+}} \sqrt{n_{+}}} \nonumber \\
			&\leq \tilde{C} \frac{e^{\left(NR^2+\delta_N\right)(1-\log(\frac{NR^2+\delta_N}{NR^2}))-NR^2}}{\sqrt{n_{+}}} \leq \frac{\tilde{C}e^{-\frac{\delta_N^2}{NR^2} +O\left(\frac{\delta_N}{N}\right)}}{R\sqrt{N}} .
		\end{align}
		
		Combining \eqref{edge kernel decomposition}, \eqref{E1 estimate}, \eqref{E2 estimate} and \eqref{E3 estimate} we conclude 
		\begin{equation*}
			\tilde{K}^{J_0}_{\mathcal{Y}}(z,w) = \frac{e^{-R\cdot \Re(z+w)}}{\pi \left(z+\overline{w}\right)^2} \left[1 - \left(\frac{R^2-1+c}{R}(z+\overline{w})+1\right) e^{-\frac{R^2-1+c}{R}\left(z+\overline{w}\right)}\right] + O\left( \frac{\log^2 (N)}{N}\right) .
		\end{equation*}
		
		Thus by \eqref{edge kernels similarity}, we conclude that with probability $ 1 - O(N^{-10})$, 
		\begin{equation*}
			\tilde{K}^J_{\mathcal{Y}}(z,w) = \frac{e^{-R\cdot \Re(z+w)}}{\pi \left(z+\overline{w}\right)^2} \left[1 - \left(\frac{R^2-1+c}{R}(z+\overline{w})+1\right) e^{-\frac{R^2-1+c}{R}\left(z+\overline{w}\right)}\right] + O\left( \frac{\log^2 (N)}{N}\right) .
		\end{equation*}
		
		Thus the distribution of $x_i = y_i \cdot \frac{R^2-1+c}{R}  $ is a mixture of determinantal point processes, that with probability at least $1-O(N^{-10}) $ have kernels
		
		\begin{align*}
			\left(\frac{R}{R^2-1+c}\right)^2 &\tilde{K}^J_{\mathcal{Y}}\left(\frac{R}{R^2-1+c}z,\frac{R}{R^2-1+c}w\right) \\
			&= \frac{e^{-R\cdot \Re(z+w)} \left[1-(z+\overline{w}+1)e^{-z-\overline{w}}\right]}{\pi \left(z+\overline{w}\right)^2}  + O\left( \frac{\log^2 (N)}{N}\right) .
		\end{align*}
		
		Those kernels are equivalent to $	K_{\mathcal{X}}(z,w) + O\left( \frac{\log^2 (N)}{N}\right) $ when we recall that
		\begin{equation*}
			K_{\mathcal{X}}(z,w) = \frac{1-(z+\overline{w}+1)e^{-z-\overline{w}}}{\pi \left(z+\overline{w}\right)^2} .
		\end{equation*}
		
		As $ f$ is a bounded function with compact support we conclude, by Lemma~\ref{determinantal process expectation bound},
		\begin{equation*}
			E_{f,N} = E_{\mathcal{X},f} + O\left(\frac{\log(N)^2}{N}\right) .
		\end{equation*}
		
	\end{proof}

	\begin{appendices}	
		
	\section{Incomplete Gamma Function} \label{incomplete gamma appendix}
	
	We use the following definitions of the incomplete gamma functions and the normalised incomplete gamma function.
	\begin{Def}
		The incomplete gamma functions are defined by 
		\begin{align*}
			\Gamma(a,s) = \int_{s}^{\infty} x^{a-1}e^{-x}\mathrm{d}x, && 			\gamma(a,s) = \int_{0}^{s} x^{a-1}e^{-x}\mathrm{d}x  .
		\end{align*} 
	
	\end{Def}
	
	\begin{Def}
		The normalised incomplete gamma function is 
		$Q(a,z)=\frac{\Gamma(a,z)}{\Gamma(a)}$.
	\end{Def}
	
	The incomplete gamma function satisfy the following recursion formula.
	
	\begin{Lemma} \label{incomplete gamma recursion}
		for $ a>0$: 
		\begin{equation*}
			\Gamma(a+1,s)= a\Gamma(a,s) + s^a e^{-s} .
		\end{equation*} 
	\end{Lemma}

	We will use the following estimates from \cite[Lemma~2.2 and Lemma~2.4]{LargeGapAsymptotics}, and see also \cite[Formulas~8.11.2,8.11.3]{DLMF} and \cite[Section~11.2.4]{SpecialFunctionsBook}.
	\begin{Lemma}
		Let $a >0 $ be fixed. As $z\to \infty$,
		\begin{equation*}
			\Gamma(a,z)= O(e^{-\frac{z}{2}}) .
		\end{equation*}
	\end{Lemma}
	
	\begin{Lemma} \label{Q representation}
		The following holds 
		\begin{equation*}
			Q(a,z) =\frac{1}{2}\operatorname{erfc}(-\eta\sqrt{\nicefrac{a}{2}}) -R_a(\eta)
		\end{equation*}
		where $\operatorname{erfc}$ is the error function, $\lambda = \frac{z}{a}$, $\eta = (\lambda-1)\sqrt{\frac{2(\lambda-1-\log \lambda)}{(\lambda-1)^2}} $ and $R_a$ satisfies, as $a\to \infty$ and uniformly in $ z\in [0,\infty)$, 
		\begin{equation*}
			R_a(\eta) \sim \frac{e^{-\frac{1}{2} a\eta^2}}{\sqrt{2\pi a}} \sum_{j=0}^{\infty} \frac{c_j(\eta)}{a^j}
		\end{equation*}
		where $c_j$ are bounded functions on $\mathbb{R}$ (i.e bounded for $z\in [0,\infty)$).
	\end{Lemma}
	
	In addition we have the following estimate for the error function (see \cite[7.12.1]{DLMF}) .
	\begin{Lemma} \label{erfc asymptotics}
		The error function satisfies, as $y\to \infty$, 
		\begin{equation*}
			\operatorname{erfc}(y) = \frac{e^{-y^2}}{\sqrt{\pi}} \left( \frac{1}{y} - \frac{1}{2y^3} +O(y^-5)\right)
		\end{equation*}
		and $\operatorname{erfc} (-y) = 2-\operatorname{erfc}(y)$.
	\end{Lemma}
	By combining Lemma~\ref{Q representation} and Lemma~\ref{erfc asymptotics} we conclude
	
	\begin{Lemma} \label{Incomplete Gamma estimates}
		\begin{enumerate}
			\item Let $\delta>0 $ be fixed. As $a\to \infty , \lambda \to \lambda_0 \geq 1$ ,such that $ (\lambda-1) \sqrt{a} \to \infty $ ,
			\begin{equation*}
				\frac{\Gamma(a,\lambda a)}{\Gamma(a)} = \frac{e^{-a (\lambda -1 -\log(\lambda))}}{\sqrt{2\pi}} \frac{1}{\lambda -1 } \frac{1}{\sqrt{a}}  \left(1 +O\left(\frac{1}{a} + \frac{1}{a(\lambda-1)^2}\right)\right) .
			\end{equation*}
			   
			\item As $ a\to \infty,  \lambda \to \lambda_0 > 0 $, such that $ (1-\lambda)\sqrt{a} \to \infty $,
			\begin{equation*}
				\frac{\Gamma(a,\lambda a)}{\Gamma(a)} = 1 - \frac{e^{-a (\lambda -1 -\log(\lambda))}}{\sqrt{2\pi}} \frac{1}{\lambda -1 } \frac{1}{\sqrt{a}} \left(1 + O\left(\frac{1}{a} + \frac{1}{a(\lambda-1)^2}\right)\right) .
			\end{equation*}
		\end{enumerate}
	\end{Lemma}

	In addition we will use the following facts for integer values of $ a$.
	\begin{Lemma} [See {\cite[Section 8.4.10]{DLMF}}] \label{series rep of incomplete}
		
		For integer $n$ and positive $z$
		\begin{equation*}
			Q(n,z) = e^{-z}\sum_{k=0}^{n-1} \frac{z^k}{k!} .
		\end{equation*}
	\end{Lemma}
	
	As a consequence we get the following lemma.
	\begin{Lemma} \label{Incomplete gamma monotonicity}
		For integer values of $n$ and positive $z$
		
		\begin{equation*}
			Q(n,z) < Q(n+1,z) .
		\end{equation*}
		
	\end{Lemma}

	\section{The Partition Function} \label{partition function}
	
	The partition function $p(n) $ is the number of different way to write $n$ as a sum of positive integers, up to summation order. 
	The asymptotics of the partition function is discussed for example in \cite{PartitionFunctionBook}, we will use only a basic claim, proved here for completeness.
	
	\begin{Lemma} \label{partition function bound}
		For any constant $ k>1 $, for $n $ big enough the partition function satisfy $ p(n) \leq k^n $.
	\end{Lemma}
	
	\begin{proof}
			We look at the function $ f(q)=\frac{1}{\prod_{j=1}^{\infty}(1-q^j) } $. Since it is analytic in the unit disk its Taylor expansion converges there, but as its Taylor expansion is $ f(q)=\sum_{n=0}^{\infty} p(n) q^n $ we conclude that $ \limsup_{n \to \infty} \sqrt[n]{p(n)} \leq 1 $.
			Hence, for any $ k>1 $, for all except finitely many values of $n$ we have $ p(n) \leq k^n $.
	\end{proof}
	
	\end{appendices}
	
	\printbibliography[heading=bibliography]
	
\end{document}